\documentclass[journal]{IEEEtran}
\usepackage[utf8]{inputenc} 
\usepackage[english]{babel}
\usepackage{chngcntr}
\usepackage{mathrsfs}
\usepackage{amsmath,amssymb}
\usepackage{mathtools}
\usepackage{algorithm}
\usepackage{algorithmicx}
\usepackage{algpseudocode}
\usepackage{blkarray}
\usepackage{bm}
\usepackage[margin=2.5cm]{geometry}%
\usepackage{xcolor}
	\definecolor{ocre}{RGB}{0,173,239}
	\definecolor{mygreen}{RGB}{28,172,0} 
	\definecolor{mylilas}{RGB}{170,55,241}
	\definecolor{mygray}{RGB}{255,255,255}
	\definecolor{myblue}{RGB}{0, 143 , 218}
\usepackage{import}
\usepackage[hidelinks]{hyperref} 
\usepackage[sort]{cite} 
\usepackage{makeidx}
\usepackage{nameref} 
\usepackage{enumitem} 
\usepackage{comment}
\usepackage{etoolbox}
\makeatletter
	\patchcmd{\@setref}{\bfseries ??}{{\color{red}[\texttt{\detokenize{ #3 }}]}}{}{%
	  \GenericWarning{}{Failed to patch \protect\@setref}}
	\patchcmd{\@citex}{\bfseries ?}{{\color{red}[\texttt{\detokenize{ #3 }}]}}{}{%
	  \GenericWarning{}{Failed to patch \protect\@citex}}
\makeatother
\usepackage{graphicx}
\usepackage{float}
\usepackage{pgfkeys}
\usepackage{epstopdf} 
	\DeclareGraphicsExtensions{.pdf,.png,.jpg,.eps,.ps}
\usepackage{multirow}	
\usepackage{subcaption}
\def \IMG {images}
\usepackage{tikz}
\usepackage{hf-tikz}
\usetikzlibrary{ 
	calc,
	automata, 
	shapes, 
	snakes, 
	positioning, 
	decorations,
	fit,
	matrix
	}
	\tikzstyle{noeud-std}=[draw,fill=black,circle,inner sep=0pt,minimum size=7pt]
\usepackage{tkz-graph}

\usetikzlibrary{arrows,matrix,positioning,fit}
\tikzset{%
  highlight/.style={rectangle,rounded corners,fill=ocre!50,draw,
    fill opacity=0.5,thick,inner sep=0pt}
}

\usepackage{amsthm}

\newtheorem{theorem}{Theorem}[section]

\newtheorem{proposition}[theorem]{Proposition}
\newtheorem{lemma}[theorem]{Lemma}

\newtheorem{corollary}[theorem]{Corollary}
\newtheorem{definition}[theorem]{Definition}
\newtheorem{remark}[theorem]{Remark}

\newtheorem{example}[theorem]{Example}

\DeclareMathAlphabet\mathbfcal{OMS}{cmsy}{b}{n}

\newcommand{\mset}{\mathcal{A}}

\newcommand{\vset}{\mathcal{V}}

\newcommand{\Smat}{\mathbf{S}}
\newcommand{\Dmat}{\mathbf{D}}

\newcommand{\U}{\mathbf{U}}

\newcommand{\pclf}{\text{pclf}}

\newcommand{\reels}{\mathbb{R}}

\newcommand{\graph}{\mathbfcal{G}}

\newcommand{\naturals}{\mathbb{N}}

\newcommand{\alphabet}[1]{\langle#1\rangle}

\newif\ifdraft
\drafttrue 
\ifdraft
\providecommand{\com[2]}{\begin{tt}[#1: #2]\end{tt}}
\else 
\providecommand{\com[2]}{\begin{tt}\end{tt}}
\fi

\begin{document}
\title{On Path-Complete Lyapunov Functions: Geometry and Comparison}
%
%
%

\author{Matthew Philippe, Nikolaos Athanasopoulos, David Angeli and Rapha\"{e}l M. Jungers%
\thanks{D. Angeli is affiliated both to the Dept. of Electrical and Electronic Engineering at the Imperial College London, UK, and to the Dept. of Information Engineering, University of Florence, Italy. E-mail: d.angeli@imperial.ac.uk.}%
\thanks{R.M. Jungers and M. Philippe are with the ICTEAM institute of the  Universit\'{e} catholique de Louvain, Belgium. They are supported by  the French Community
of Belgium (ARC grant 13/18-054) and by the SESAR 2020 project COPTRA. M.P. is a F.N.R.S.-FRIA Fellow. E-mails:
\{raphael.jungers,  matthew.philippe\}@uclouvain.be}%
\thanks{N. Athanasopoulos is with the School of Electronics, Electrical Engineering and Computer Science, Queen's University Belfast, Belfast, Northern Ireland. E-mail: n.athanasopoulos@qub.ac.uk}
}

\markboth{Journal of \LaTeX\ Class Files,~Vol.~, No.~, August~2017}%
{Shell \MakeLowercase{\textit{et al.}}: Bare Demo of IEEEtran.cls for IEEE Journals}
%

\maketitle

\begin{abstract}
We study optimization-based criteria for the stability of switching systems, known as \emph{Path-Complete Lyapunov Functions,}  and ask the question ``can we decide algorithmically when a criterion is less conservative than another''.
Our contribution is twofold. First, we show that a Path-Complete Lyapunov Function, which is a \emph{multiple} Lyapunov function by nature,
 can always be expressed as a  \emph{common} Lyapunov function taking the form of a combination of minima and maxima of the elementary functions that compose it. 
Geometrically, our results provide for each Path-Complete criterion an implied invariant set.
 Second, we provide a linear programming criterion allowing to compare the conservativeness of two arbitrary given Path-Complete Lyapunov functions.
\end{abstract}
\begin{IEEEkeywords}
Path-Complete Methods, Lyapunov stability theory, Conservativeness, Automata, Switching Systems.
\end{IEEEkeywords}
\IEEEpeerreviewmaketitle

\section{Introduction}
\IEEEPARstart{S}{witching} systems
\cite{LiSISA,ShWiSCFS,JuTJSR,LiAnSASO} present major theoretical challenges 
\cite{BlTsTBOA}. They provide an accurate modeling framework for many  processes
\cite{MaAlCOHC,HeMiOAMS,ShWiAPSM, DoHeSAON,HeDaEBTL} and 
can be used as abstractions for more complex hybrid dynamical systems \cite{GiPaABAB}. 
We focus on discrete-time linear switching systems, with the following dynamics:
\begin{equation}
x(t+1) = A_{\sigma(t)} x(t).
\label{intro:eq:swsys}
\end{equation}
There, at any time $t$, $\sigma(t) \in \{1, \ldots, M\}$ is the \emph{mode} of the system and each mode corresponds to a matrix from a set of $M$ matrices $\mset = \{A_1, \ldots, A_M\}$. 
We call a sequence of modes  $\sigma(0)\sigma(1)\ldots$ a \emph{switching sequence}.

The question of the stability of a switching system has been a major challenge in the Control Engineering literature in the past decades \cite{LiMoBPIS, LiAnSASO, ShWiSCFS}. We are interested in the study of certificates for the stability under \emph{arbitrary switching}.
\begin{definition}
The System \eqref{intro:eq:swsys} is \emph{stable under arbitrary switching} if there is $K \in \reels$ such that
for \emph{any} switching sequence $\sigma(0)\sigma(1)\ldots$, where $\sigma(t) \in \{1, \ldots, M\}$, the
trajectories satisfy
$$ \forall x(0) \in \reels^n, \, \forall t \in \naturals: \|x(t)\| \leq K \|x(0)\|.$$
The System is asymptotically stable if it is stable and furthermore, for any switching sequence and initial condition, $\lim_{t \rightarrow \infty}\|x(t)\| = 0$.
\end{definition} 
The problem of deciding whether or not a switching system is stable under arbitrary switching is difficult, and in general undecidable (see e.g. \cite{JuTJSR,BlTsTBOA}).
Nevertheless, several tools have been developed, which provide \emph{semi-algorithms} to decide \emph{asymptotic stability} \cite{PhEsSODT,BlFeSAOD,LeDuUSOD,EsLeCOLS,MoPyCOAS, BlMiSTMI,AnLaASCF, PaJaAOTJ}. 

A popular approach to assess stability for switching systems is to look for a \emph{common Lyapunov function (CLF)}.
The method is attractive because stable arbitrary switching systems always have a CLF (see e.g. \cite[Theorem 2.2]{JuTJSR}). However,  in general, whatever the technique used to search for such a function, if it is tractable, it can only provide conservative stability certificates. The search for a \emph{common quadratic Lyapunov function} (see e.g. \cite[Section II-A]{LiAnSASO}) illustrates this fact well.
There, the goal is to find a \emph{positive definite quadratic function} $V : \reels^n \rightarrow \reels_{\geq 0} : x \mapsto x^\top Q x,$ for a positive definite matrix $Q \succ 0$, such that
\begin{equation}
 \forall \sigma \in \{1, \ldots, M\}, \forall x \in \reels^n: V(A_\sigma x) \leq V(x). 
 \label{eq:lyapIneqCLF}
 \end{equation}
 Checking for the existence of such a function can be done efficiently using convex optimization tools because the \emph{Lyapunov inequalities} \eqref{eq:lyapIneqCLF} are equivalent to a set of \emph{linear matrix inequalities}. Nevertheless, such a Lyapunov function may not exist, even for asymptotically stable systems, see e.g.
\cite{LiMoBPIS, LiAnSASO}, and Example \ref{ex:ExampleCLF1} below. In order to alleviate this conservativeness, one may rely on more complex parameterizations for the Lyapunov function $V$ at the cost of greater computational efforts (e.g.\cite{PaJaAOTJ} uses sum-of-squares polynomials,
\cite{GoHuDMII} uses max-of-quadratics Lyapunov functions,  and reachability analysis \cite{AnLaASCF,BlMiSTMI}).

 \emph{Multiple Lyapunov functions} (see \cite{BrMLFA,ShWiSCFS,JoRaCOPQ}) arise as an alternative to the search of
common Lyapunov functions. Here again, multiple \emph{quadratic} Lyapunov  functions such as those introduced in \cite{BlFeSAOD,DaRiSAAC,
LeDuUSOD, EsLeCOLS} hold special interest because checking for their existence also amounts  to
 solve a set of linear matrix inequalities. For example, let us consider a switching System \eqref{intro:eq:swsys} on $M = 2$ modes.
The multiple Lyapunov function proposed in \cite{DaBePDLF} is composed of two positive definite quadratic functions $V_a, V_b : \reels^{n} \rightarrow \reels_{\geq 0}$
 that satisfy the following sets of inequalities $\forall x \in \reels^n$:
\begin{equation}
\begin{aligned}
V_{a}(A_1x)& \leq V_{a}(x), \\
 V_{b}(A_1x)& \leq V_{a}(x), \\
 V_{a}(A_2x)& \leq V_{b}(x), \\
 V_{b}(A_2x)& \leq V_{b}(x).
\end{aligned}
\label{eq:4LyapunIneq}
\end{equation}

These tools  make use of convex optimization and linear matrix inequalities in order to provide powerful algorithms for the stability analysis of switching Systems. 
In order to further analyze such tools, Ahmadi et al. recently introduced the concept of \emph{Path-Complete Lyapunov functions \cite{AhJuJSRA}}.
There, multiple Lyapunov functions such as the one of \cite{DaBePDLF} mentioned above are represented by \emph{directed and labeled graphs}, see Figure  \ref{figure:FirstPCGraphs}.\\
 More precisely, let $\graph = (S,E)$ be a graph where $S$ is the set of nodes, and $E \subset S \times S \times \{1, \ldots, M\}$ is the set of directed edges labeled by one of the $M$ modes\footnote{In a more general setting, the labels on the edges can be taken as finite sequences of modes. Our results extend there through the so-called \emph{expanded graph} \cite[Definition 2.1]{AhJuJSRA}.} of the System \eqref{intro:eq:swsys}. To each node $s \in S$ of the graph, we assign one positive definite quadratic function  $V_s : \reels^n \rightarrow \reels_{\geq 0}$. An edge $(s,d,\sigma) \in E$ then encodes the following inequality:
 \begin{equation}
\forall x \in \reels^n: V_d(A_\sigma x ) \leq V_s(x).
\label{eq:lyapIneqEdge}
\end{equation}
This formalism provides a framework under which to unify, generalize and study multiple Lyapunov functions such as cited above. Fundamental properties of a multiple Lyapunov functions represented by a graph $\graph$ can be deduced from the properties of that graph. In particular, in \cite{AhJuJSRA, JuAhACOL} the authors provide a characterization of the set of graphs that \emph{represent stability certificates for switching systems on $M$ modes}. This property, known as \emph{Path-Completeness}, leads to the concept of \emph{Path-Complete Lyapunov functions} (see Definition \ref{def:PCLF} below).

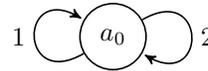
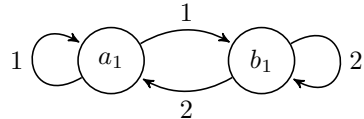
\begin{figure}[!ht]
\centering
\begin{subfigure}[c]{0.4\textwidth}
\centering
\begin{tikzpicture}[->,>=stealth',shorten >=1pt,auto,node distance=2.8cm,
                   semithick, scale = 1, transform shape ]
 \node[state] (V1)                           {$a_0$};

 \path (V1)edge [loop right, out = 30, in = -30, min distance = 10mm] node {$2$}             (V1)
     	   edge [loop left, out = 210, in = 150, min distance = 10mm]  node {$1$}             (V1)
     	   ;
\end{tikzpicture}
\caption{Graph $\graph_0$, corresponding to the inequalities \eqref{eq:lyapIneqCLF}, for a common Lyapunov function. }
\label{fig:PCExample_0}
\end{subfigure}
\begin{subfigure}[c]{0.4\textwidth}
\centering
\begin{tikzpicture}[->,>=stealth',shorten >=1pt,auto,node distance=2cm,
                    semithick, scale = 1, transform shape ]
 \node[state] (V1)                           {$a_1$};
 \node[state] (V2)     [right of = V1]       {$b_1$};
  \path (V1)edge [loop,  out = 210, in = 150, min distance = 10mm]            node {$1$} (V1)
      	    edge [bend left] node {$1$}            (V2)
     	(V2)edge [bend left]             node {$2$} (V1)
      	    edge [loop,  out = 30, in = -30, min distance = 10mm] node {$2$} (V2);
\end{tikzpicture}
\caption{Graph $\graph_1$, corresponding to the inequalities \eqref{eq:4LyapunIneq}, leading to the multiple Lyapunov function of \cite{DaBePDLF} for a system on two modes. }
\label{fig:PCExample_1}
\end{subfigure}
\caption{ Examples of labeled and directed graphs.  }
\label{figure:FirstPCGraphs}
\end{figure}

Several challenges exist in the study of Path-Complete Lyapunov functions, in particular for matters related to how these certificates compare with each other with respect to their conservativeness\footnote{Preliminary versions of our results have been presented in \cite{AnPhPCGC,AnAtALPT}.}.
In this paper we first ask a natural question which aims at revealing the connection to classic Lyapunov theory:\\
\textbf{Q1: }\emph{Can any Path-Complete Lyapunov function be represented as a common Lyapunov function}?\\  We answer this question affirmatively in Theorem \ref{thm:EmbeddedCLF}. We show that if a System \eqref{intro:eq:swsys} has a PCLF for a graph $\graph = (S,E)$ and with functions $V_s, s \in S$, the system has a common Lyapunov  function of the form
\begin{equation}
 V(x)  = \min_{S_1, \ldots, S_k \subseteq S} \left ( \max_{s \in S_i} V_s(x) \right ),
 \label{eq:embededLyapunovFunction}
\end{equation}
for some finite integer $k$, where the sets $S_i$ are subsets of the nodes of $\graph$.
Our proof is constructive and makes use of a classical tool from automata theory, namely the \emph{observer automaton} \cite{CaLaITDE}.
We then discuss in Subsection \ref{subsec:converse} the conservativeness of these common Lyapunov functions when using quadratic Path-Complete Lyapunov functions and argue that it is ultimately linked with the combinatorial nature of the graph itself.\\
Motivated by this, we provide in Section \ref{sec:Comparison} answers to the following question:\\
\textbf{Q2: }\emph{When does one graph lead to systematically less conservative stability certificates than another?}\\ 
 We say that a graph $\graph$ is \emph{more conservative} than a graph $\graph'$ if for any set of matrices $\mset$, the solvability of the LMIs corresponding to $\graph$ implies that of the LMIs for $\graph'$. We provide an algorithmic sufficient condition in Theorem \ref{thm:comparisonLP},  inspired by existing ad-hoc proofs for particular cases such as the ones presented in \cite[Section 4 and 5]{AhJuJSRA}, \cite[Theorem 3.5]{PhEsSODT}, \cite[Theorem 20]{GeGiAACL}. \\
Finally, in Section \ref{sec:DiscussionAndConclussion}, we comment on our results and conclude our work.

\begin{remark}
For the clarity of exposition, and because this is by far the most popular case, we restrict the presentation to linear switching systems under arbitrary switching and consider Path-Complete Lyapunov functions with quadratic functions as pieces. Our results can be generalized in several directions, e.g. to more general classes of pieces (continuous, positive definite and radially increasing) or to \emph{constrained switching systems} \cite{PhEsSODT}.
\end{remark}

\section{Preliminaries}
\label{sec:prel}

Given an integer $M$, we let $\alphabet{M}$ denote the set $\{1,\ldots,M\}$. Given a discrete set $X$, we let $|
X|$ denote the cardinality of the set.\\
We now define the central concept of this paper (see Figure \ref{intro:fig:graphs} for illustrations).
\begin{definition}[{Path-Completeness}]
A graph $\graph = (S,E)$ is \emph{Path-Complete} if for any $k \geq 1$ and any sequence $\sigma =
\sigma_1\ldots,\sigma_k$,  $\sigma_i \in \alphabet{M}$, there is a \emph{path} in 
the graph $(s_i, s_{i+1}, \sigma_i)_{i = 1, 2,
\ldots, k }$ with $(s_i, s_{i+1}, \sigma_i) \in E$.
\label{intro:def:PC}
\end{definition}
\begin{figure}[!ht]
\centering
\begin{subfigure}[c]{0.4\textwidth}
\centering
\begin{tikzpicture}[->,>=stealth',shorten >=1pt,auto,node distance=2.5cm,
                    semithick, scale = 1, transform shape ]
 \node[state] (a)                          {$a_2$};
 \node[state] (b)     [left of = a]       {$b_2$};
 \node[state] (c)     [right of = a]       {$c_2$};

  \path (a) edge [bend right] node [above] {$1$} (b)
  			edge node[above]{$2$} (b)
  			 edge [bend left] node {$1$} (c)
  			edge node{$2$} (c)
  		(b) edge [bend right] node [below] {$1$} (a)
  		(c) edge [bend left] node {$2$} (a)
  	    	;
\end{tikzpicture}
\caption{Path-Complete graph $\graph_2$. It corresponds to a multiple Lyapunov function with 3 functions
$V_{a_2}, \, V_{b_2}, \, V_{c_2}$ satisfying 6 Lyapunov inequalities.}
\label{fig:PCExample_2}
\end{subfigure}
\begin{subfigure}[c]{0.4\textwidth}
\centering
\begin{tikzpicture}[->,>=stealth',shorten >=1pt,auto,node distance=2.5cm,
                    semithick, scale = 1, transform shape ]
 \node[state] (a)                          {$a_3$};
 \node[state] (b)     [left of = a]       {$b_3$};
 \node[state] (c)     [right of = a]       {$c_3$};

  \path (a) edge [bend right] node [above] {$1$} (b)
  			edge node[above]{$2$} (b)
  			 edge [bend left] node {$1$} (c)
  		(b) edge [bend right] node [below] {$1$} (a)
  		(c) edge [bend left] node {$2$} (a)
  	    	;
\end{tikzpicture}
\caption{Graph $\graph_3$. It is not Path-Complete as it cannot generate the sequence $222$. }
\label{fig:PCExample_3}
\end{subfigure}
\caption{ Illustration for Definition \ref{intro:def:PC}. }
\label{intro:fig:graphs}
\end{figure}
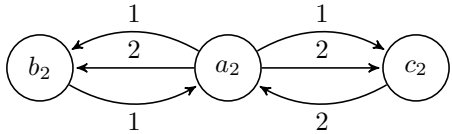
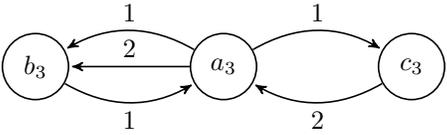

As said above, a Path-Complete Lyapunov function is a multiple Lyapunov function where Lyapunov inequalities between a set of \emph{quadratic positive definite functions} are encoded in a labeled and directed graph $\graph = (S,E)$, with one function per node in $S$.
We represent a set of  such functions in a vector form. 
\begin{definition}[\textbf{{VLFC}}]
A \emph{Vector Lyapunov Function Candidate (VLFC)} is a vector function $V : \reels^n \rightarrow \reels^N_{\geq 0}$,
where each element $V_i: \reels^n \mapsto \reels_{\geq 0}$, $i\in \alphabet{N}$ is a positive definite quadratic function.
\end{definition}
Given a graph $\graph = (S,E)$ and a VLFC $V : \reels^n \rightarrow \reels^{|S|}_{\geq 0}$ we let $V_s$ be the function for the node $s \in S$.

\begin{definition}[Path-Complete Lyapunov function (PCLF)]
Given a switching System \eqref{intro:eq:swsys} on a set of $M$ matrices $\mset$ of dimension $n$ and a \emph{Path-Complete graph} $\graph = (S,E)$ on $M$ labels, a \emph{Path-Complete Lyapunov function} for $\graph$ and $\mset$ is a VLFC $V : \reels^n \rightarrow \reels^{|S|}_{\geq 0}$ such that
\begin{equation}
\forall x \in \reels^n, \forall (s,d,\sigma) \in E: V_d(A_\sigma x) \leq V_s(x).
\label{eq:feasi1}
\end{equation}
We write $V \sim \pclf(\graph, \mset) $
to denote the fact that $V$ is a Path-Complete Lyapunov function for $\graph$ and $\mset$.
\label{def:PCLF}
\end{definition}

\section{Extracting Common Lyapunov Functions}
\label{sec:PCLFtoCLF}

In  this section we are given a Path-Complete graph $\graph$, a set of matrices $\mset$, and a Path-Complete Lyapunov function $V \sim \pclf(\graph, \mset)$, and we want to construct a common Lyapunov function for the switching system on the set $\mset$. In order to do so, we provide an algorithm relying on concepts from \emph{Automata theory} (see e.g. \cite[Chapter 2]{CaLaITDE}). 
Preliminary versions of our results in this section have been presented in the conference paper \cite{AnPhPCGC}.

\subsection{Main Result}

Our main result exploits the structure of Path-Complete graphs to combine the functions of a PCLF into a common Lyapunov function for the System \eqref{intro:eq:swsys}. 
The following proposition is the first step to achieve this.

\begin{proposition} 
Consider a Path-Complete graph $\graph = (S,E)$ on $M$ labels, a set of $M$ matrices $\mset$, and a PCLF $V \sim \pclf(\graph, \mset)$.
 Take two subsets $P$ and $Q$ of $S$. 
\begin{itemize}
 \item If there is a label $\sigma$ such that
\begin{equation}
\forall p \in P, \,\exists q \in Q: \, (p,q,\sigma) \in E,
\label{eq:left-total_relation}
\end{equation}
 then
\begin{equation}
\forall x \in \reels^n: \min_{q \in Q} V_q(A_\sigma x) \leq \min_{p \in P}{V_p(x)}.
\label{eq:left-total_relation:min}
\end{equation}
\item If there is a label $\sigma$ such that
\begin{equation}
\forall q \in Q, \,\exists p \in P: \, (p,q,\sigma) \in E,
\label{eq:right-total_relation}
\end{equation}
 then
\begin{equation}
\forall x \in \reels^n:  \max_{q \in Q} V_q(A_\sigma x) \leq \max_{p \in P}{V_p(x)}.
 \label{eq:right-total_relation:max}
 \end{equation}
\end{itemize} 
\label{prop:hiddenEq:complete}
\end{proposition}
We propose a geometric proof for the above.
\begin{proof}
Take a graph $\graph = (S,E)$ with $M$ labels, a set of $M$ matrices $\mset$, and a PCLF $V \sim \pclf(\graph, \mset)$. 
For any node $s \in S$,  define the one-level set 
$$X_{s} = \{x \in \reels^n \mid V_s(x) \leq 1\}.$$
For any edge $(s,d,\sigma) \in E$, \eqref{eq:lyapIneqCLF} is equivalent to 
$ A_\sigma X_{s} \subseteq X_{d},$
where $A_\sigma X_s = \{A_\sigma x \mid x \in X_s\}$. 
In light of this, \eqref{eq:left-total_relation} is equivalent to 
$$  \forall p \in P, \exists q \in Q: A_\sigma X_{p} \subseteq X_{q}, $$
and from there it is easy to conclude that
$$   A_\sigma \bigcup_{p \in P}   X_{p} = \bigcup_{p \in P} A_\sigma  X_{p} \subseteq \bigcup_{q \in Q} X_{q}. $$
This is then equivalent to \eqref{eq:left-total_relation:min} since the union of the level sets of the functions $V_q, q \in Q$ is the level set of the function $\min_{q \in Q} V_q$. \\
Similarly, \eqref{eq:right-total_relation} is
$$ \forall q \in Q, \exists p \in P: A_\sigma X_{p} \subseteq X_{q}$$
which in turn implies 
$$A_\sigma \bigcap_{p \in P} X_{p}  = \bigcap_{p \in P}  A_\sigma X_{p}  \subseteq  \bigcap_{q \in Q}  X_{q} , $$
which is equivalent to \eqref{eq:right-total_relation:max} since the the intersection level sets of the functions $V_q, q \in Q$ is the level set of the function $\max_{q \in Q} V_q$.
\end{proof}

The above proposition can already be put to good use to extract common Lyapunov functions from Path-Complete Lyapunov functions where the graph is either \emph{complete or co-complete} (see \cite{SaEOAT} Definition 1.12).
\begin{definition}[(Co)-Complete Graph]
A graph  $\graph = (S,E)$ is \emph{complete} if for all $s \in S$, for all $\sigma \in \alphabet{M}$ there exists at least one edge $(s, q, \sigma) \in E$.\\
The  graph is \emph{co-complete} if for all $q \in S$, for all $\sigma \in \alphabet{M}$, there exists at least one edge $(s, q, \sigma) \in E$.
\label{def:(co)complete}
\end{definition}
Note that the graph $\graph_1$ in Figure \ref{fig:PCExample_1} is co-complete.
On can check that if a graph is complete (or co-complete), it is Path-Complete as well.
Proposition \ref{prop:hiddenEq:complete} has the following corollary.
\begin{corollary}
Consider a graph $\graph$ on $M$ modes, a set of $M$ matrices $\mset$, and a PCLF $V  \sim \pclf(\graph, \mset)$. 
\begin{itemize}
\item If $\graph$ is complete, then
$ \bar{V}(x) = \min_{s \in S} V_s(x) $
is a common Lyapunov function for System \eqref{intro:eq:swsys}.
\item If $\graph$ is co-complete, then $ \bar{V}(x) = \max_{s \in S} V_s(x) $
is a common Lyapunov function for System \eqref{intro:eq:swsys}.
\end{itemize}
\label{cor:CLF(co)complete}
\end{corollary}
\begin{example}

We consider the following switching system consisting of $M=2$ modes: $x({t+1})  =  A_{\sigma(t)} x(t)$, $\sigma(t) \in \{1,2\} $, with 
\begin{equation}
A_1 = \alpha \begin{pmatrix}  1.3 & 0 \\ 1 & 0.3 \end{pmatrix} , A_2 = \alpha \begin{pmatrix} -0.3 & 1 \\ 0 & -1.3\end{pmatrix},
\label{eq:exampleSystem}
\end{equation}
with $\alpha = (1.4)^{-1}$.
This system does not have a common quadratic Lyapunov function.
However, for the graph $\graph_1$ represented in Figure \ref{fig:PCExample_1}, we have the following Path-Complete Lyapunov function:
\begin{equation}
\vset = \left \{
\begin{aligned}
&V_{a_1} (x) = 5x_1^2+x_2^2,\\
& V_{b_1}  (x) = x_1^2+5x_2^2 
\end{aligned}\right \}.
\label{eq:ex1Solution}
\end{equation}
The graph $\graph_1$ being co-complete, the function $\bar{V}(x) = \max (V_{a_3}(x), V_{b_3}(x))$ is a common Lyapunov function for the system.
This is represented in Figure \ref{fig:lvlSets}, where we see that the intersection of the level sets of $V_{a_3}$ and $V_{b_3}$, which is that of $\bar{V}$, is itself invariant. 
\begin{figure}
\centering
\input{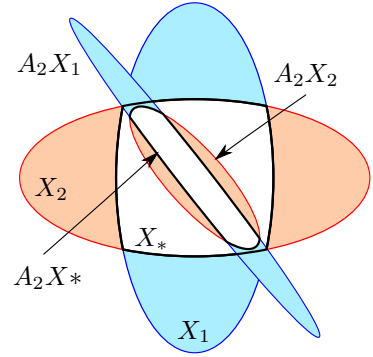}
\caption{
Example \ref{ex:ExampleCLF1}: Geometric representation of the Path-Complete stability criterion corresponding to the graph $\graph_1$ in Figure \ref{fig:PCExample_1}. The ellipsoids $X_1$ and $X_2$ are the level sets of the quadratic functions $V_{a_1}$ and $V_{b_1}$ from Example \ref{ex:ExampleCLF1} respectively, and $X_*$ is the level set of $\max(V_{a_1}(x), V_{b_1}(x))$.  The set $X_*$ is invariant, as illustrated by the fact that the set $A_2 X_*$ is in $X_*$. }
\label{fig:lvlSets}
\end{figure}
\label{ex:ExampleCLF1}
\end{example}

In order to tackle Path-Complete graphs that are neither complete nor co-complete (as in Figure \ref{fig:PCExample_4}), we introduce  the following concept. 
\begin{definition}[Observer Graph,  {\cite[Section 2.3.4]{CaLaITDE}}]
Consider a graph $\graph = (S,E)$. The \emph{observer} graph $\graph^{obs} = (S^{obs}, E^{obs})$ is a graph where each state corresponds to a subset of $S$, i.e. $S^{obs} \subseteq 2^S$, and is constructed as follows:
\begin{itemize}
\item[1.] Initialize $S^{obs} := \{ S \}$ and $E^{obs} := \emptyset$.
\item[2.] Let $X:=\emptyset$. For each pair $(P,\sigma) \in S^{obs} \times \alphabet{M}$: 
\begin{enumerate}[label=(\roman*)]
\item Compute $Q := \underset{p \in P}{\cup} \{q | \, (p,q,\sigma) \in E\}.$
\item If $Q \neq \emptyset$, set $E^{obs}:=E^{obs}\cup\{(P,Q,\sigma)\}$ then  $X:=X\cup Q$.
\end{enumerate}
\item[3.] If  $X \subseteq S^{obs}$, then the observer is  given by $\graph^{obs} = (S^{obs}, E^{obs})$. Else, let
 $S^{obs} := S^{obs} \cup X$ and go to step 2.
\end{itemize}
\label{def:observers}
\end{definition}

\begin{example}
Consider the graph $\graph_4$ in Figure \ref{fig:PCExample_4}.
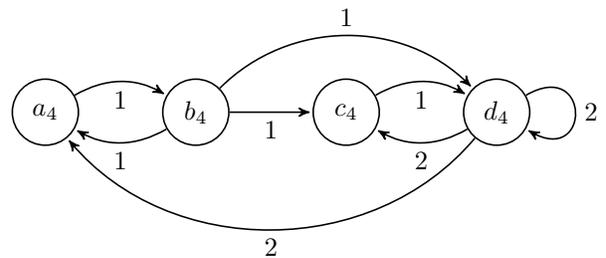
\begin{figure}[!ht]
\centering
\begin{tikzpicture}[->,>=stealth',shorten >=1pt,auto,node distance=2cm,
                    semithick, scale = 1, transform shape ]
 \node[state] (a)                          {$a_4$};
 \node[state] (b)     [right of = a]       {$b_4$};
 \node[state] (c)     [right of = b]       {$c_4$};
 \node[state] (d)     [right of = c]       {$d_4$};

  \path (a) edge [bend left]        node [below]{$1$} (b)
  	    (c) edge [bend left] node [below]{$1$} (d)
  	    (b) edge [bend left] node [below]{$1$} (a)
  	    	edge  node [below]{$1$} (c)
  	    	edge [bend left = 45]node [above] {$1$} (d)
  	    (d) edge [bend left = 50]node{$2$} (a)
  	    	edge [bend left] node{$2$} (c)
  	    	edge [loop,  out = 30, in = -30, min distance = 10mm] node{$2$} (d)
  	    	;
\end{tikzpicture}
\caption{Path-Complete graph $\graph_4$ for Example \ref{ex:observer}.}
\label{fig:PCExample_4}
\end{figure}
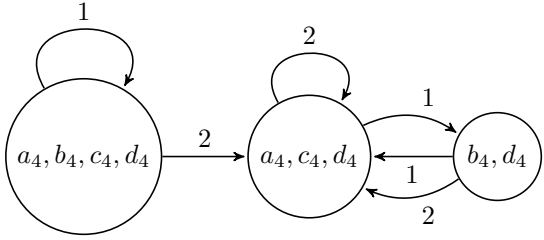
\begin{figure}[!ht]
\centering
\begin{tikzpicture}[->,>=stealth',shorten >=1pt,auto,node distance=2cm,
                    semithick, scale = 1, transform shape ]
 \node[state] (abcd)                          {$a_4,b_4,c_4,d_4$};
 \node[state] (acd)     [node distance = 3cm, right of = abcd]       {$a_4,c_4,d_4$};
 \node[state] (bd)     [node distance = 2.5cm, right of = acd]       {$b_4,d_4$};

  \path (abcd) edge [loop,  out = 120, in = 60, min distance = 12mm] node {$1$} (abcd)
  				edge node {$2$} (acd)
  		(acd) edge [loop,  out = 120, in = 60, min distance = 10mm] node {$2$} (acd)
  			  edge [bend left] node {$1$} (bd)
  	    (bd) edge node {$1$} (acd)
  	    	edge [bend left] node{$2$} (acd)
  	    	;
\end{tikzpicture}
\caption{Observer graph $\graph_4^{obs}$ of the graph $\graph_4$ in Figure \ref{fig:PCExample_4}. The nodes of $\graph_4^{obs}$ are associated to sets of nodes of $\graph_4$. }
\label{fig:PCExample_4_Obs}
\end{figure}
The observer graph $\graph^{obs}_4$ is given in Figure \ref{fig:PCExample_4_Obs}. The first run through step 2 in Definition \ref{def:observers} is as follows. We have $P = S$. For $\sigma = 1$ the set $Q$ is again $S$ itself: indeed, each node $s \in S$ has at least one inbound edge with the label $1$. For $\sigma = 2$, since node $b_4$ has no inbound edge labeled $2$, we get $Q = \{a_4,c_4,d_4\}$. This set is then added to $S^{obs}$ in step 3, and the algorithm repeats step 2 with the updated $S^{obs}$.\\
\label{ex:observer}
\end{example}

\begin{remark}
The \emph{observer automaton} is presented in \cite[Section 2.3.4]{CaLaITDE}. 
Our definition of observer graph is an adaption in the particular case where the automaton considered has all states marked both as starting and accepting states.
\label{rem:GraphAndAutomata}
\end{remark}

We are now in position to introduce the main result of this section.

\begin{theorem}[CLF Representation of a PCLF]
Consider a set of $M$ matrices $\mset$ and a Path-Complete graph $\graph$. 
If there is a PCLF $V \sim \pclf(\graph, \mset)$, 
then a common Lyapunov function for System \eqref{intro:eq:swsys} is given by
\begin{equation}
V^*(x) = \min_{Q \in S^{obs}} \left ( \max_{s \in Q} V_s(x) \right ),
\label{eq:minmaxclf}
\end{equation}
where $S^{obs}$ is the set of nodes of the observer graph $\graph^{obs} = (S^{obs}, E^{obs})$ of the graph $\graph$.
\label{thm:EmbeddedCLF}
\end{theorem}

\begin{proof}
Consider a Path-Complete graph $\graph = (S,E)$, a set of matrices $\mset$, and a PCLF $V \sim \text{pclf}(\graph, \mset)$.  Construct the \emph{observer graph} $\graph^{obs} = (S^{obs},E^{obs})$. By construction, there is an edge $(P,Q,\sigma) \in E^{obs}$ if and only if $Q = \cup_{p \in P} \{q | \, (p,q,\sigma) \in E\}$. This means that $\forall q \in Q$, $\exists p \in P$ such that $(p,q,\sigma)\in E$, which is \eqref{eq:right-total_relation}.
Consequently, from Proposition \ref{prop:hiddenEq:complete}, we have that
$$
\begin{aligned}
 (P,Q,\sigma) \in E^{obs} & \Rightarrow \\
  & \forall x \in \reels^n : \max_{q \in Q} V_q( A_\sigma x) \leq \max_{p \in P}{V_p(x)}. 
\end{aligned}
$$ 
We now consider a set of functions $\tilde{V} = \{ \tilde{V}_Q, Q \in S^{obs}\}$ where $\tilde{V}_Q = \max_{q \in Q} V_q$, and observe that
$$ \forall (P,Q,\sigma) \in E^{obs}, \forall x \in \reels^n: \tilde{V}_Q(A_\sigma x) \leq \tilde{V}_P(x), $$
which means that these  functions satisfy all Lyapunov inequalities encoded in  $\graph^{obs}$. 

Our next step is now to make use of Proposition \ref{prop:hiddenEq:complete} and Corollary \ref{cor:CLF(co)complete} to show that the pointwise-minimum of these functions is a CLF for the system. \\
To achieve this, first observe that the proof of Proposition \ref{prop:hiddenEq:complete} holds verbatim when instead of having a PCLF where each entry $V_s$ is quadratic, we take each entry to be \emph{a pointwise-maximum of quadratics} as in the above.\\
 Second, we claim that the \emph{observer graph of a Path-Complete graph} is \emph{complete}.
We do this by contraposition: if $\graph^{obs} = (S^{obs}, E^{obs})$ is not complete, then it must be so that $\graph$ is not Path-Complete. 
We emphasize that the nodes of an observer graph correspond to sets of nodes in the original graph, and refer to them as thus.\\
Assume that there is one set of nodes $P \in S^{obs}$, $P \subsetneq S$, and a label $\sigma^* \in \alphabet{M}$ such that there are no edges $(P,Q,\sigma) \in E^{obs}$. By construction of $\graph^{obs}$, there are directed paths from the node $S \in S^{obs}$ (corresponding to the full set of nodes in $\graph$) to the node $P$ above. Take any of such paths, and let $\sigma_1, \ldots, \sigma_k$ be the sequence of labels on that path. By the definition of the observer graph, we know that $P$ contains all the nodes of $\graph$ that are the destination of some path carrying that sequence of labels. 
By our choice of $P$, we know that none of these nodes have an outgoing edge with the label $\sigma^*$. Otherwise, there would be an edge $(P,Q,\sigma*) \in E^{obs}$ for some $Q$. We conclude that the sequence $\sigma_1 \ldots \sigma_k \sigma^*$ can not be found on a path in $\graph$, and $\graph$ is therefore not Path-Complete.\\
 Since $\graph^{obs}$ is complete, we can use Corollary \ref{cor:CLF(co)complete} to deduce that the function 
$$ V^* = \min_{Q \in S^{obs}} \tilde{V}_Q $$
is a common Lyapunov function for the system \eqref{intro:def:PC} on the set $\mset$, which concludes the proof.
\end{proof}

\begin{example}
Consider the following set of matrices taken from \cite[Example 11]{GoHuDMII}:
\begin{equation}
\mset = \left \{\alpha \begin{pmatrix}
  0.3 & 1 & 0 \\
  0 & 0.6 & 1\\
  0 & 0 & 0.7
\end{pmatrix}, \alpha \begin{pmatrix}
0.3 & 0 & 0 \\
-0.5 & 0.7 & 0 \\
-0.2 & -0.5 & 0.7
\end{pmatrix} \right \}
\label{eq:matricesForExample}
\end{equation}\\
 with the choice of $\alpha = 1.03$. 
 That switching system has a Path-Complete Lyapunov function $V$ for the graph $\graph_4 = (S,E)$ in Figure \ref{fig:PCExample_4}.
 After inspecting the observer graph $\graph^{obs}$,
 we compute a common Lyapunov function for the system as the minimum of the two functions $\max \left ( V_{a_4}, V_{c_4}, V_{d_4}\right)$ and $\max \left ( V_{b_4}, V_{d_4}\right)$. The last term $\max_{s \in S} \left (V_s) \right ) $ can be omitted here.
Figure \ref{fig:evolutionInTimeOfPCLF}  shows the evolution in time of the four functions, and that of the common Lyapunov function, from the initial condition $x(0) = \begin{pmatrix}
0 & 0 & -1
\end{pmatrix}^\top$, and for the \emph{periodic} switching sequence repeating the pattern $21111$.  
\begin{figure}[!ht]
\centering
\includegraphics[width = \columnwidth]{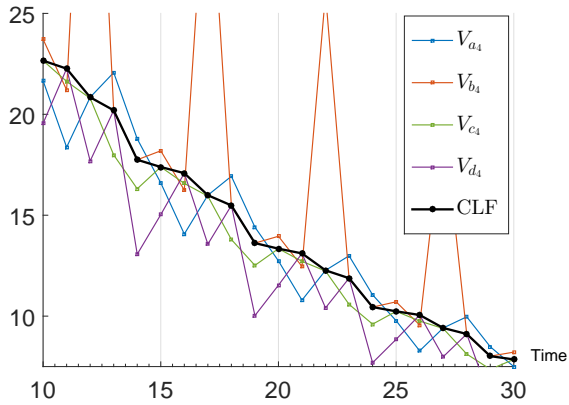}
\caption{Evolution of a PCLF for the system in Example \ref{exa:ExamplePCLFEvolves} for the graph $\graph_4$ in Figure \ref{fig:PCExample_4}. Observe that none of the pieces of the PCLF decreases monotonically. However, the common Lyapunov function (in black) does decrease monotonically.}
\label{fig:evolutionInTimeOfPCLF}
\end{figure}
\label{exa:ExamplePCLFEvolves}
\end{example}

\subsection{Comparison with classical piecewise quadratic Lyapunov Functions.}
\label{subsec:converse}

Our results highlight a link between Path-Complete Lyapunov functions and common Lyapunov functions that are mins or max of functions. It is thus natural to ask whether or not \emph{any} Lyapunov function of the form (\ref{eq:embededLyapunovFunction}) can be induced from a Path-Complete graph with as many nodes as the number of pieces of the function itself.
We give a negative answer to this question, under the form of a counter example. To do so, we use the system in \cite[Example 11]{GoHuDMII}, where min-of-quadratic and max-of-quadratic Lyapunov functions have been studied.

\begin{example}

Consider the switching systems on the two modes with the matrices presented in \eqref{eq:matricesForExample} for $\alpha = 1$.
The system has a max-of-quadratics Lyapunov function $x \mapsto \max\{V_1(x),V_2(x) \}$,
with $V_i(x) = \left ( x^\top   Q_i x \right )$, where
$$
\begin{aligned}
& Q_1 =
\begin{pmatrix}
  36.95 & -36.91 & -5.58 \\
   -36.91 & 84.11 & -38.47\\
  -5.58 & -38.47 & 49.32
\end{pmatrix}, \\
& Q_2 = \begin{pmatrix}
  13.80 & -6.69 & 4.80 \\
  -6.69 & 21.87 & 10.11\\
  4.80  &  10.11 & 82.74
\end{pmatrix}.
\end{aligned}
$$
These functions are obtained by solving a set of Bilinear Matrix Inequalities (BMIs) (see \cite[Section 5]{GoHuDMII}).
However, for the same system, we cannot obtain a Path-Complete Lyapunov function that can then be represented as a max-of-quadratics with two pieces.
More precisely, we considered all the graphs $\graph = (S,E)$ that are co-complete on two nodes. There are 16 of them in total. For each of them, the convex optimization program  corresponding to the search of a Path-Complete Lyapunov function has no solutions.
\label{example:BMI}
\end{example}
The example above involves two approaches for trying to compute a max-of-quadratics Lyapunov function. The first approach relies on solving a set of BMIs.
As pointed out in \cite{GoHuDMII}, these BMIs can be hard to solve in general (there is no polynomial-time algorithm to solve them).
  In contrast, searching for a quadratic Path-Complete Lyapunov function can be done efficiently by solving linear matrix inequalities using convex optimization tools. We can therefore efficiently check if a max-of-quadratics common Lyapunov function corresponding to a PCLF exists. Nevertheless, as one can see in Example \ref{example:BMI}, the BMI approach can be less conservative than the approach using PCLFs. We leave for further work the question of understanding whether or not the BMI approach is less conservative than the PCLF approach in general.

Additionally,  among the 16 Path-Complete graphs tested in Example \ref{example:BMI}, some lead to more conservative stability certificates than others. For example,  some of these graphs have, at one of their two nodes, two self loops with different labels, as it is the case for the graph $\graph_0$ in Figure \ref{fig:PCExample_0}. If a Path-Complete Lyapunov Function is found for such a graph, then  the function for that node is itself a common quadratic Lyapunov function for the system,  and thus that graph is at least as conservative as the common quadratic Lyapunov function technique. 
However, there are pairs of graphs for which none of them is more conservative than the other.  This fact is illustrated in the following example.

\begin{example} 
Consider the three graphs $\graph_1$ in Figure \ref{fig:PCExample_1},  $\graph_5$ and $\graph_6$ in Figure \ref{fig:twoOtherGraphs}.
These graphs are co-complete.
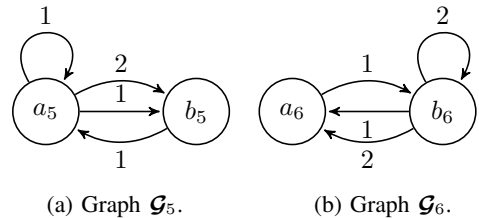
\begin{figure}[!ht]
\centering
\begin{subfigure}[c]{0.2\textwidth}
\centering
\begin{tikzpicture}[->,>=stealth',shorten >=1pt,auto,node distance=2cm,
                    semithick, scale = 1, transform shape ]
 \node[state] (V1)                           {$a_5$};
 \node[state] (V2)     [right of = V1]       {$b_5$};
 
  \path (V1)edge [loop,  out = 120, in = 60, min distance = 10mm]   node {$1$} (V1)
			edge             node {$1$} (V2)
      	    edge [bend left] node {$2$} (V2)
      (V2)    edge [bend left] node {$1$}            (V1);

\end{tikzpicture}
\caption{Graph $\graph_5$.  }
\label{fig:PCExample_5}
\end{subfigure}~
\begin{subfigure}[c]{0.2\textwidth}
\centering
\begin{tikzpicture}[->,>=stealth',shorten >=1pt,auto,node distance=2cm,
                    semithick, scale = 1, transform shape ]
 \node[state] (V1)                           {$a_6$};
 \node[state] (V2)     [right of = V1]       {$b_6$};
 
  \path (V2)edge [loop,  out = 120, in = 60, min distance = 10mm]   node {$2$} (V2)
			edge             node {$1$} (V1)
      	    edge [bend left] node {$2$} (V1)
      (V1)    edge [bend left] node {$1$}            (V2);
 
\end{tikzpicture}
\caption{Graph $\graph_6$.   }
\label{fig:PCExample_6}
\end{subfigure}
\caption{ Two co-deterministic Path-Complete graphs.  }
\label{fig:twoOtherGraphs}
\end{figure}

The system on two modes with matrices
 \begin{equation}
\mset = \left \{\alpha
\begin{pmatrix}
-0.5 & -1.1 \\
0.9 & 1.5
\end{pmatrix}, 
\alpha
\begin{pmatrix}
0.2 & 1.0 \\
0.5 & 0.5
\end{pmatrix}
\right  \},
\end{equation}
with $\alpha = (1.05)^{-1}$, has a quadratic PCLF for $\graph_1$ in Figure \ref{fig:PCExample_1} but neither for $\graph_5$ nor $\graph_6$ in Figures \ref{fig:PCExample_5} and \ref{fig:PCExample_6}.
The system on two modes with matrices
 \begin{equation}
\mset = \left \{\alpha
\begin{pmatrix}
0 & -0.2 \\
0.8 & 0
\end{pmatrix}, 
\alpha
\begin{pmatrix}
0.25 & 0.4 \\
0.1 & 0.3
\end{pmatrix}
\right  \},
 \end{equation}
with $\alpha = (0.55)^{-1}$ has a quadratic PCLF for the graph $\graph_6$, but neither for $\graph_3$ or $\graph_5$. The same set of matrices, when swapping the two modes, will have a PCLF for $\graph_5$, but not for the two other graphs.
\label{ex:notTotal}
\end{example}

The fact that PCLFs can be represented as common Lyapunov functions  closes a gap between multiple and common Lyapunov functions techniques. However, the discussion above illustrates that this does not directly help us to compare the conservativeness of PCLFs. Indeed, we see that several PCLFs lead to CLFs with similar structures, but even among these PCLFs it remains non-trivial to decide which ones are more conservative than others.
For this reason, in the next section, we provide novel tools for comparing the conservativeness of graph-Lyapunov functions based on their graphs.

\section{Comparing graphs}
\label{sec:Comparison}

Quadratic Path-Complete Lyapunov functions are particularly attractive in practice since given a graph $\graph = (S,E)$ and a set of matrices $\mset$, one can verify their existence  by solving a set of LMIs.  More precisely, we search for quadratic forms $\{Q_s \succ 0, s \in S\}$ satisfying the matrix inequalities
\begin{equation}
\forall (s,d,\sigma) \in E: A_\sigma^\top Q_d A_\sigma \preceq Q_s.
\label{eq:LMIsineq}
\end{equation}
Our goal in this section is to provide systematic tools to decide when, given two graphs $\graph$ and $\graph'$,  it is true that \emph{for any set of matrices  $\mset$}, the existence of a quadratic graph-Lyapunov function for $\graph$ implies that of a quadratic GLF for $\graph'$. 

This relates to the setting of  \cite{AhJuJSRA}, where Path-Complete Lyapunov functions with quadratic pieces are used for the \emph{approximation of the exponential growth rate}, a.k.a. the \emph{joint spectral radius  \cite{RoStANOT, JuTJSR}}, of switching systems. 
More precisely, for each graph $\graph = (S,E)$ with labels in $\alphabet{M}$, and for any set of $M$ matrices $\mset$ we let 
\begin{equation}
\begin{aligned}
\gamma(\graph,\mset) & = \inf_{Q, \gamma} \gamma : \\
& \qquad \forall (s,d,\sigma) \in E: A_\sigma^\top Q_d A_\sigma^{} \preceq \gamma^{2} Q_s, \\
& \qquad \forall s \in  S: Q_s \succ 0. \\
\end{aligned}
\label{chV:eq:CJSRApprox}
\end{equation}
We can then capture the fact that a graph leads to more conservative stability certificate than another by the following \emph{ordering relation}: given two Path-Complete graphs $\graph$ and $\graph'$ with labels in $\alphabet{M}$:
\begin{equation}
\begin{aligned}
& \graph \leq \graph' \text{ if } \forall n \in \naturals, \forall \mset \subset \reels^{n \times n}:\\
& \gamma(\graph,\mset) \geq \gamma(\graph', \mset).
\end{aligned}
\label{eq:ordering}
\end{equation}
As pointed out in Example \ref{ex:notTotal}, this above relation does not form a total order, since the graphs $\graph_1$, $\graph_5$ and $\graph_6$ are \emph{incomparable}.

The techniques we develop herein are inspired by existing ad-hoc proofs in the literature applying to particular cases/sets of graphs. Examples of these ad-hoc proofs can be found in \cite[Proposition 4.2 and Theorem 5.4]{AhJuJSRA}, \cite[Theorem 3.5]{PhEsSODT}. Typically, given two graphs $\graph$ and $\graph'$, these proofs proceed in two steps to show that $\graph \leq \graph'$. The first step is to propose a construction to transform any possible quadratic PCLF for $\graph$ into a candidate quadratic PCLF for $\graph'$. 
The second step is to check that the construction does indeed provide a PCLF for $\graph'$. 
\begin{example}
Take any set $\mset$ of two matrices , and let $V \sim \pclf(\graph_2, \mset)$ be a Path-Complete Lyapunov function for the graph $\graph_2 = (S_2,E_2)$ in Figure \ref{fig:PCExample_2}. Then, one can show that a Path-Complete Lyapunov function for the graph $\graph_1 = (S_1, E_1)$ in Figure \ref{fig:PCExample_1}  is given by $U = \left (U_{a_1}, U_{b_1} \right )^{\top}$, with
\begin{equation}
U_{a_1} = V_{a_2} + V_{b_2}, \, U_{b_1} = V_{a_2} + V_{c_2} .
\label{eq:Cexample}
\end{equation}
To prove this we need to show that each edge $(s,d,\sigma) \in E_2$ represents a valid Lyapunov inequality for the functions in $U$. For example, consider the edge $(a_1,b_1,2) \in E_1$. 
We need to show that 
$$ \forall x \in \reels^n: V_{a_2}(A_2 x) + V_{c_2}(A_2 x) \leq V_{a_2}(x) + V_{b_2}(x)$$ 
holds true. To do so, it suffices to write down and sum up the inequalities corresponding to the edges $(a_2,b_2,2)$ and $(c_2,b_2,2)$ in $E_2$, which are assumed to hold true for the choice of functions $V$. The reader can verify that a similar reasoning can be applied to all edges in $E_1$. Therefore, for any set of matrices $\mset$, as long as $V \sim \pclf(\graph_1, \mset)$, the set $U$ defined above satisfies $U \sim \pclf(\graph_1, \mset)$. 
\label{example:ordering0}
\end{example}

In this section, we focus on providing algorithmic tools to decide whenever a construction as in Example \ref{example:ordering0} exists. We formalize this as follows.

\begin{definition}
Consider two Path-Complete graphs $\graph = (S,E)$ and $\graph' = (S',E')$ on the same labels $\alphabet{M}$. We write 
$$ \graph \leq_\Sigma \graph' $$
if there is a matrix\footnote{The element $C_{s',s}$ appearing in \eqref{eq:conics} is the element for the row of $s' \in S'$ and the column of $s \in S$ of $C \in \reels^{|S'| \times |S|}_{\geq 0}$.}  $C \in \reels^{|S'| \times |S|}_{\geq 0}$, satisfying
\begin{equation}
\forall s' \in S': \sum_{s \in S} C_{s',s} \geq 1,
\label{eq:conics}
\end{equation}
such that for any set of $M$ matrices $\mset$ of dimension $n$ and Path-Complete Lyapunov function $V \sim \pclf(\graph, \mset)$, the VLFC $U : \reels^n \rightarrow \reels^{|S'|}_{\geq 0}$
where $$ \forall s' \in S', \forall x \in \reels^n: U_{s'}(x) = \sum_{s \in S} C_{s',s}V_s(x) $$
satisfies $U \sim \pclf(\graph', \mset)$.
\label{def:conicOrdering}
\end{definition}

Given two graphs, the property $\graph \leq_\Sigma \graph'$ is a \emph{sufficient condition} for the ordering $\graph \leq \graph'$ \eqref{eq:ordering}. This is due to the fact that the set of quadratic functions is closed under positive combinations.

The following allows us to express  \eqref{eq:feasi1} with vector inequalities.

\begin{definition}
Given a graph $\graph = (S,E)$ with a set of labels $\alphabet{M}$, and $\sigma \in \alphabet{M}$,  we define the two matrices $\Smat^\sigma(\graph) \in \{0,1\}^{|E_\sigma| \times |S|}$ and $\Dmat^\sigma(\graph) \in \{0,1\}^{|E_\sigma| \times |S|}$ as follows:
\begin{equation}
\begin{split}
& \Smat^{\sigma}_{e,s} = 1 \Leftrightarrow \exists d \in S:  e = (s, d, \sigma) \in E, \\
&  \Dmat^{\sigma}_{e,d} = 1 \Leftrightarrow \exists s \in S: e = (s, d, \sigma) \in E,
\end{split}
\label{eq_PQ}
\end{equation}
where $E_\sigma \subset E$ is the set of edges  with label~$\sigma$. 
\end{definition}
The construction of these matrices is illustrated in Example \ref{ex:SmatDmat}.
 For a graph $\graph$ and label $\sigma$, the matrix $\Smat^{\sigma}(\graph) - \Dmat^{\sigma}(\graph)$ is the incidence matrix \cite{CaLaITDE} of the subgraph of $\graph$ having only the edges with label $\sigma$.\\
Given a set of $M$ matrices $\mset$ of dimension $n$ and a Path-Complete graph $\graph = (S,E)$ on $M$ labels, a VLFC $V : \reels^n \rightarrow \reels^{|S|}_{\geq 0}$ satisfies $V \sim \pclf{\graph, \mset}$ if and only if
\begin{equation}
\forall \sigma \in \alphabet{M}, \forall x \in \reels^n: \Dmat^\sigma(\graph)V(A_\sigma x) \leq \Smat^\sigma(\graph)V(x),
\label{eq:feasVLFC}
\end{equation}
where the vector inequality is taken entrywise.\\
Our main result in this section is the following theorem, whose proof is detailed in Subsection \ref{proofOfMainThm}.

\begin{theorem}
Consider two graphs $\graph = (S,E)$ and $\graph' = (S',E')$ on $M$ labels.
The following statements are equivalent.
\begin{itemize}
\item[a)] The graphs satisfy $\graph \leq_\Sigma \graph'$.
\item[b)] There exists a matrix $C \in \reels^{|S'| \times |S|}_{\geq 0}$ satisfying \eqref{eq:conics} and one matrix $K_\sigma \in \reels^{|E^{'}_{\sigma}|, |E_{\sigma}|}_{\geq 0}$ per label $\sigma \in \alphabet{M}$ such that 
\begin{align}
&\Smat^\sigma(\graph')C \geq K_\sigma \Smat^\sigma(\graph), \label{eq:theLP-S} \\
&\Dmat^\sigma(\graph')C \leq K_\sigma \Dmat^\sigma(\graph).\label{eq:theLP-D}
\end{align}
\end{itemize}
\label{thm:comparisonLP}
\end{theorem}

\begin{example}
\label{ex:SmatDmat}
In this example, we apply Theorem \ref{thm:comparisonLP} to show that $\graph_1 \leq_{\Sigma} \graph_2$, where $\graph_1$ is represented in Figure \ref{fig:PCExample_1} and $\graph_2$ is represented in Figure \ref{fig:PCExample_2}.\\
We first construct the matrices $\Smat^\sigma$, and $\Dmat^\sigma$ for the graphs $\graph_1$ and $\graph_2$ and $\sigma \in \alphabet{2}$. 
In order to do so, we need to set a convention for ordering nodes and edges in a graph (to convene to which edge and node corresponds the entry $\Smat^\sigma_{1,1}(\graph_2)$ for example).
In any graph represented in this paper, we let the first node is the node marked with ``$a_g$'', the second node is that marked with ``$b_g$'', etc..., where the subscript $g$ designates the graph itself.
We use a lexicographical ordering for the edges, and sort them according to their source node first, then their destination node, and finally their labels. For example, the first edge in $\graph_4$ in Figure \ref{fig:PCExample_4} is the edge $(a_4, b_4, 1)$, and its third edge would be the edge $(b_4, c_4, 1)$.\\
With these conventions, we obtain the following matrices:
$$
\Smat^1(\graph_1) = \begin{pmatrix}
1 & 0 \\
1 & 0
\end{pmatrix}, \,
\Dmat^1(\graph_1) = \begin{pmatrix}
1 & 0 \\
0 & 1  
\end{pmatrix},
$$
$$
\Smat^2(\graph_1) = \begin{pmatrix}
0 & 1 \\
0 & 1
\end{pmatrix}, \,
\Dmat^2(\graph_1) = \begin{pmatrix}
1 & 0 \\
0 & 1  
\end{pmatrix},
$$
$$
\Smat^1(\graph_2) = \begin{pmatrix}
1 & 0&  0 \\
1 & 0 & 0\\
0& 1 & 0 
\end{pmatrix}, \,
\Dmat^1(\graph_2) = \begin{pmatrix}
0 & 1&  0 \\
0 & 0 & 1\\
1& 0 & 0 
\end{pmatrix},
$$
$$
\Smat^2(\graph_2) = \begin{pmatrix}
1 & 0&  0 \\
1 & 0 & 0\\
0& 0 & 1 
\end{pmatrix}, \,
\Dmat^2(\graph_2) = \begin{pmatrix}
0 & 1&  0 \\
0 & 0 & 1\\
1& 0 & 0 
\end{pmatrix}.
$$

For these choices, a solution to the inequalities \eqref{eq:theLP-S}, \eqref{eq:theLP-D} is given by
$$ C = \begin{pmatrix}
1 &1& 0 \\
1& 0 & 1
\end{pmatrix},  
K_1 = K_2 =  \begin{pmatrix}
1 & 0& 1\\
0 & 1 &1
\end{pmatrix}.
$$

\end{example}

\subsection{From algebraic to linear inequalities}
\label{proofOfMainThm}
The main challenge for the proof of Theorem \ref{thm:comparisonLP} is to show that $a) \Rightarrow b)$. With our matrix/vector notations, 
Definition \ref{def:conicOrdering} can be written as follows: given two graphs $\graph = (S,E)$ and $\graph' = (S', E')$,
$ \graph  \leq_\Sigma \graph'  $ if there is a matrix $C \in \reels^{|S'|\times |S|}_{\geq 0}$ satisfying \eqref{eq:conics} such that
$$ \forall \mset, \forall V \sim \pclf(\graph, \mset): U = C V \sim \pclf(\graph', \mset).$$

In this subsection, we show that these algebraic conditions are equivalent to the \emph{inclusion of one polyhedral cone into another}. We begin by investigating the range of values that may be taken by a VLFC $V : \reels^n \rightarrow \reels^{|S|}_{\geq 0}$, independently of the dimension $n$.
\begin{lemma}
Take any pair of integers $M, N \geq 1$ and any \emph{positive} vector $\lambda \in \reels^{(M+1)N}_{> 0}$. There exists a point $u \in \reels^{M+1}$,
 a set of matrices $\mathcal{T} = \{T_1, \ldots, T_M\}$ in  $\reels^{M+1 \times M+1}$, and a VLFC 
$U : \reels^{M+1} \rightarrow \reels^{N}_{\geq 0}$ such that 
\begin{equation}
\lambda = \begin{pmatrix}
\lambda^0 \\
\lambda^1 \\
\vdots \\
\lambda^M
\end{pmatrix} =  \begin{pmatrix} \U(u) \\ \U(T_1u) \\ \vdots \\ \U(T_M u)\end{pmatrix},
\label{eq:lambdaAsAVLFC}
\end{equation}
where for $0 \leq i \leq M$, the block $\lambda^i$ is a positive vector of dimension $N$.\\
If furthermore for a graph $\graph = (S,E)$ with $|S| = N$, the vector $\lambda$ satisfies
\begin{equation}
\forall \sigma \in \alphabet{M}: \Smat^\sigma(\graph)\lambda^0 \geq  \Dmat^\sigma(\graph)\lambda^\sigma,
 \label{eq:LambdaIsHAppy.}
\end{equation} 
then we can pick the VLFC $U$ and the matrices $\mathcal{T}$  such that $U \sim \pclf(\graph, \mathcal{T})$. 
\label{lemma:posishort}
\end{lemma}
\begin{proof}
We begin with the first part of the Lemma.
For all $k \in \{1, \ldots, N\}$ we define the quadratic function $U_k : \reels^{M+1} \rightarrow \reels_{\geq 0}$ as follows:
\begin{equation}
 U_k(x) = \sum_{i = 0}^M \lambda^i_k x_{i+1}^2. 
 \label{eq:UVLFC}
\end{equation}
We now take the vector $u \in \reels^{M+1}$ such that $u_1 = 1$ and $u_i = 0$ for $i \geq 2$. Clearly, $U_k(u) = \lambda^0_k$, so we have $U(u) = \lambda^0$ by definition. \\
For the matrices $\mathcal{T}$, we take for each $\sigma \in \alphabet{M}$ the matrix $T_{\sigma}$ defined as follows:
\begin{equation}
 \forall 1\leq i,j \leq M+1, T_{\sigma, {i,j}} = \delta_{i,\sigma + 1}\delta_{j,1},
 \label{eq:TransitionMaps}
 \end{equation}
 where for $1 \leq k,\ell \leq M+1$, $\delta_{k,\ell} = 1$ if $k = \ell$ and $\delta_{k,\ell} = 0$ else.
We can then verify that for all $\sigma \in \alphabet{M}$, $U(T_\sigma u) = \lambda^\sigma.$ This concludes the proof of the first part.

Let us now take a graph $\graph = (S,E)$ with $N$ nodes and $M$ labels, and assume \eqref{eq:LambdaIsHAppy.} holds for a given label $\sigma$.
This means that $\forall (s,d,\sigma) \in E$, $(\lambda^0)_s \geq (\lambda^\sigma)_d.$
Take any  edge $(s,d,\sigma) \in E$, and any $x \in \reels^{(M+1)}$.
Taking again $U$ as in \eqref{eq:UVLFC}, we have
$$
\begin{aligned}
U_s(x) & = \sum_{i = 0}^M\lambda^i_s \, x_{i+1}^2  \\ & \geq  \lambda^0_s \, x_{1}^2 \geq \lambda^\sigma_d \, x_1^2 = U_d(T_\sigma x).
\end{aligned}
 $$
Since this holds at all edges of the graph with label $\sigma$, and all $x \in \reels^{(M+1)}$, we can conclude the proof.
\end{proof}

\begin{proposition}\label{prop:OrdToPointwise}
Consider two graphs $\graph=(S,E)$ and $\graph'=(S',E')$ on $M$ labels and a matrix $C\in\reels_{\geq 0}^{|S'|\times |S|}$ satisfying~\eqref{eq:conics}.
The following statements are equivalent.
\begin{itemize}
\item [a)] The graphs satisfy $\graph \leq_\Sigma \graph'$ with the matrix $C$.
\item [b)] For any set of $M$ matrices $\mset = \{A_1, \ldots, A_M\}$ in any dimension $n$, for any VLFC $V: \reels^n \rightarrow \reels^{|S|}_{ \geq 0},$ for any $\sigma \in \alphabet{M}$ and any $x \in \reels^n$:
\begin{align}
&\Dmat^\sigma(\graph)V(A_\sigma x) \leq \Smat^\sigma(\graph)V(x) \label{eq:WorksForXandG} \Rightarrow \\
&\Dmat^\sigma(\graph')CV(A_\sigma x) \leq \Smat^\sigma(\graph')C V(x). \label{eq:WorksForXandG'} 
\end{align}
\end{itemize}
\end{proposition}
\begin{proof}
The difference between the two statements is subtle yet important. Statement $a)$ is equivalent to \emph{``if \eqref{eq:WorksForXandG} holds for all $\sigma \in \alphabet{M}$ and all $x \in \reels^n$, then so does \eqref{eq:WorksForXandG'}"}. In contrast, statement $b)$ merely states that \eqref{eq:WorksForXandG'} holds at any point and label where \eqref{eq:WorksForXandG} holds.\\
With this in mind, proving $b) \Rightarrow a)$ is immediate as, under $b)$, if \eqref{eq:WorksForXandG} holds for all points $x \in \reels^n$ and all $\sigma \in \alphabet{M}$, so does \eqref{eq:WorksForXandG'}. \\
We prove that $a) \Rightarrow b)$ by using the construction in the proof of Lemma \ref{lemma:posishort}. Assume that $a)$ holds and consider a set of matrices $\mset$, a VLFC $V$, and one particular point $x^*$ such that \eqref{eq:WorksForXandG} holds at $x^*$. We show that \eqref{eq:WorksForXandG'} holds at $x^*$ as well. Let 
$$ \lambda = \begin{pmatrix} \lambda^0 \\ \lambda^1 \\ \vdots \\ \lambda^M \end{pmatrix} = \begin{pmatrix} V(x^*) \\ V(A_1 x^*) \\ \vdots \\ V(A_M x^*)\end{pmatrix}. $$
Without loss of generality, we can assume $\lambda$ to be strictly positive. For this vector, we know that $\Smat^\sigma(\graph) \lambda^0 \geq \Dmat^\sigma(\graph)\lambda^\sigma$, or more explicitly, $$ \forall (s,d,\sigma) \in E, \, (\lambda^0)_s \geq (\lambda^\sigma)_d.$$ Using this information, we seek to show that 
\begin{equation}
\Smat^\sigma(\graph')C \lambda^0 \geq \Dmat^\sigma(\graph')C\lambda^\sigma
\label{eq:SDLambdaStuff}
\end{equation}
 holds as well.

We now construct a vector $\tilde{\lambda} \in \reels^{N(M+1)} $ defined in blocks such that $\tilde{\lambda}^0 = \lambda^0$, $\tilde{\lambda}^\sigma = \lambda^\sigma$, and for all $\sigma' \neq \sigma$, $\tilde{\lambda}^{\sigma'} = (\min_i \lambda_i^0)\mathbf{1}$, where  $\mathbf{1} \in \reels^{|S|}$ is the vector of all ones.
By construction, \eqref{eq:LambdaIsHAppy.} holds for the vector $\tilde{\lambda}$. 
Using Lemma \ref{lemma:posishort}, we take a set of matrices $\mathcal{T} = \{T_{\sigma'}, \sigma' \in \alphabet{M}\}$ of dimension $M+1$ and a VLFC $U : \reels^{M+1} \rightarrow \reels^{|S|}_{\geq 0}$ such that there is a point $u^* \in \reels^{M+1}$ where
$$ U(u^*) = \tilde{\lambda}^0, \, \forall \sigma' \in \alphabet{M}: U(T_{\sigma'} u^*) = \tilde{\lambda}^{\sigma'}.$$   Moreover it must be that $U \sim \pclf(\graph, \mathcal{T})$ since $\eqref{eq:LambdaIsHAppy.}$ holds \emph{at all ${\sigma'} \in \alphabet{M}$}.

We can now conclude the proof using statement $a)$: since $U \sim \pclf(\graph, \mathbf{T})$, \eqref{eq:WorksForXandG} holds for $U$, $\mathcal{T}$, and all $u \in \reels^{M+1}$. Hence, so does \eqref{eq:WorksForXandG'}. In particular, it holds for the mode $\sigma$ at the point $u^*$, and we can conclude the proof since
$\Smat^\sigma(\graph')C \lambda^0 \geq \Dmat^\sigma(\graph')C\lambda^\sigma $
holds true.
\end{proof}

The next result allows us to characterize the $\graph \leq_\Sigma \graph'$ property without having to rely on concepts explicitly related to dynamics (e.g. matrix sets and VLFCs). 
\begin{proposition}\label{Prop:OrdToNonNegative}
Consider two graphs $\graph=(S,E)$ and $\graph'=(S',E')$ on $M$ labels and a matrix $C\in\reels_{\geq 0}^{|S'|\times |S|}$ satisfying~\eqref{eq:conics}.
The following statements are equivalent.
\begin{itemize}
\item [a)] The graphs satisfy $\graph \leq_\Sigma \graph'$ with the matrix $C$.
\item [b)] For any vector 
$$\lambda = \begin{pmatrix} (\lambda^0)^\top & (\lambda^1)^\top & \ldots & (\lambda^M)^\top \end{pmatrix}^\top \in \reels^{(M+1)|S|}_{\geq 0} $$
with $\lambda^i \in \reels^{|S|}_{\geq 0}$, for any $\sigma \in \alphabet{M}$:
\begin{align}
&\Dmat^\sigma(\graph)\lambda^\sigma \leq \Smat^\sigma(\graph)\lambda^0 \label{eq:WorksForXandGL} \Rightarrow \\
&\Dmat^\sigma(\graph')C\lambda^\sigma \leq \Smat^\sigma(\graph')C \lambda^0. \label{eq:WorksForXandG'L} 
\end{align}
\end{itemize}
\end{proposition}
\begin{proof}
By Proposition \ref{prop:OrdToPointwise}, we can as well prove the equivalence between the statement $b)$ above and the statement of  Proposition \ref{prop:OrdToPointwise}-b).\\
With this in mind, \emph{$b) \Rightarrow a)$} is direct. \\
For the other direction, i.e. $a) \Rightarrow b)$, we therefore need to show that Proposition \ref{prop:OrdToPointwise} - $b)$ implies the current statement $b)$. 
By Lemma \ref{lemma:posishort}, we have that any \emph{positive vector} $\lambda > 0$ satisfying \eqref{eq:WorksForXandGL} also satisfies \eqref{eq:WorksForXandG'L}. Indeed, given any such vector $\lambda$, we can construct a point $u$, a set of matrices $\mathcal{T}$ and a VLFC $U$ such that these satisfy \eqref{eq:WorksForXandG}. They must therefore satisfy \eqref{eq:WorksForXandG'}, and therefore $\lambda$ satisfies \eqref{eq:WorksForXandG'L}.\\
Finally, observe that this implies that the open polyhedral cone defined as the set $P_\sigma = \{\lambda > 0: \eqref{eq:WorksForXandGL} \text{ holds} \}  $ is included into the cone
$P_\sigma' = \{\lambda \geq 0: \eqref{eq:WorksForXandG'L} \text{ holds} \}. $ Hence the closure of  $P_\sigma$ is also in $P_\sigma'$, which is equivalent to $b)$ and concludes the proof.
\end{proof}

We are now in position to prove Theorem \ref{thm:comparisonLP}.
\begin{proof}[Proof of Theorem \ref{thm:comparisonLP}]
\emph{$b) \Rightarrow a)$:}
Given a set of $M$ matrices $\mset$ of dimension $n$, assume there is a VLFC  $V \sim \pclf(\graph, \mset)$. By definition, we have $\forall \sigma \in \alphabet{M}, \forall \, x \in \reels^n$:  $\Dmat^\sigma(\graph) V(A_\sigma x) \leq \Smat^\sigma(\graph)  V(x)$.  This implies, from \eqref{eq:theLP-S} and \eqref{eq:theLP-D}, that $\forall \sigma \in \alphabet{M}, \, \forall x \in \reels^n$:  $\Dmat^\sigma(\graph')C V(A_\sigma x) \leq \Smat^\sigma(\graph')C  V(x)$. We conclude that for any set $\mset$ such that there is $V \sim \pclf(\graph, \mset)$, it is true that $U = CV \sim \pclf(\graph', \mset)$. Therefore, $\graph \leq_\Sigma \graph'$.\\
 \emph{$a) \Rightarrow b)$:}
 Take two graphs $\graph = (S,E)$ and $\graph' = (S',E')$ and assume that $\graph \leq_\Sigma \graph'$ through some matrix $C \in \reels^{|S'| \times |S|}_{\geq 0}$ that satisfies \eqref{eq:conics}.
 Then, Proposition \ref{Prop:OrdToNonNegative} - $b)$ states the following. For all $\sigma \in \alphabet{M}$, the polyhedral sets 
\begin{align}
 & P_\sigma = \{x,y \in \reels^{|S|}_{\geq 0} \mid \Smat^\sigma(\graph) x - \Dmat^\sigma(\graph) y \geq 0  \} \label{eq:setP}, \\
 & P'_\sigma = \{x,y \in \reels^{|S|}_{\geq 0} \mid  \Smat^\sigma(\graph')C x - \Dmat^\sigma(\graph')C y \geq 0 \} \label{eq:setP}
\end{align} 
satisfy $P_\sigma \subseteq P'_\sigma$. The inclusion of polyhedral sets is well understood, and we are going to make use of the following formulation of the well-known Farkas Lemma to conclude our proof.
\begin{lemma} [Farkas Lemma, {\cite[Lemma II.2]{HeDTCL}}]
Consider two matrices $A \in \reels^{p \times n}$ and $B \in \reels^{q \times n}$. The following are equivalent:
	$$
	\begin{aligned}	
	 & (\{y \in \reels^{n}: Ay \geq \mathbf{0}, \, y \geq \mathbf{0}\} \subseteq \{y \in \reels^{n}: By \geq \mathbf{0}\}) \\ 
	 &\Leftrightarrow \exists K \in \reels^{m \times p}: KA \leq B, \, K \geq \mathbf{0}. 
	 \end{aligned}$$ 
	\label{lem:farkas}
 \end{lemma}
 Applying Lemma \ref{lem:farkas} to the inclusion $P_\sigma \subseteq P'_\sigma$, 
 there must be a matrix $K_\sigma \in \reels^{|E^{'}_{\sigma}| \times |E_\sigma|}$, where $E^{'}_{\sigma}$ and $E_{\sigma}$ are the sets of edges with label $\sigma$ in $\graph'$ and $\graph$ respectively, satisfying
\begin{equation}
 K_\sigma \begin{pmatrix}\Smat^\sigma(\graph) & - \Dmat^\sigma(\graph) \end{pmatrix} \leq \begin{pmatrix} \Smat^\sigma(\graph') & - \Dmat^\sigma(\graph') \end{pmatrix} C, 
 \label{eq:theLPinStrealinedForm}
\end{equation}
which concludes the proof.
\end{proof} 

\begin{remark}
The characterization of the ordering $\graph \leq_\Sigma \graph'$ is remarkable in its simplicity. 
In its essence, the linear program of Theorem \ref{thm:comparisonLP} mimics the proof scheme used in Example \ref{example:ordering0}. This is easier to see by inspecting \eqref{eq:theLPinStrealinedForm}. Given a graph $\graph$ on an alphabet $\alphabet{M}$, each row of the matrix $\begin{pmatrix}
\Smat^\sigma(\graph) & - \Dmat^  \sigma(\graph)
\end{pmatrix}$
corresponds to one edge of the graph (with label $\sigma$). The inequality \eqref{eq:theLPinStrealinedForm} is therefore equivalent to checking that, given a matrix $C$ mapping a PCLF for the graph $\graph$ into a PCLF for the graph $\graph'$,
we can prove that all inequalities of the graph $\graph'$ hold true simply by composing those of $\graph$. The matrices $K_\sigma$ give us the information regarding how to compose these inequalities.\\
Note that the tool recovers those from \cite[Section 4]{AnPhPCGC}, which are based on combinatorial criteria, that ultimately rely on constructing a PCLF for one graph from a PCLF for another graph by positive combination of the pieces of the first one.
 \end{remark}
\subsection{Discussion and extensions}
\label{Disc}
The approach can be naturally extended to encompass more general proofs of ordering. 
\begin{example}
Consider the graph $\graph_7$ in Figure \ref{fig:PCExample_7}.

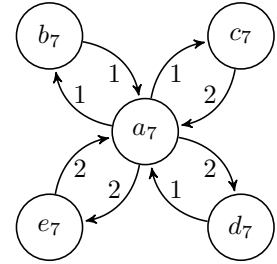
\begin{figure}[!ht]
\centering
\begin{tikzpicture}[->,>=stealth',shorten >=1pt,auto,node distance=1.8cm,
                    semithick, scale = 1, transform shape ]
 \node[state] (a)                          {$a_7$};
  \node[state] (b)  [above left of = a]                        {$b_7$};
   \node[state] (c) [above right of = a]                         {$c_7$};
    \node[state] (d)       [below right of = a]                   {$d_7$};
     \node[state] (e)      [below left of  = a ]                    {$e_7$};
     
  \path (a) edge [bend left = 35] node [above, pos = 0.4] {$1$} (b)
     		edge [bend left = 35] node [below, pos = 0.6] {$1$} (c)
     		edge [bend left = 35] node [below, pos = 0.4] {$2$} (d)
		    edge [bend left = 35] node [above, pos = 0.6] {$2$} (e)
  		(b) edge [bend left = 35] node [below, pos = 0.4]{$1$} (a)
  		(c) edge [bend left = 35] node [above, pos = 0.6] {$2$} (a)
  		(d) edge [bend left = 35] node [above, pos = 0.4] {$1$} (a)
  		(e) edge [bend left = 35] node [below, pos = 0.6] {$2$} (a);
  	    	;
\end{tikzpicture}
\caption{The Path-Complete graph $\graph_7$.}
\label{fig:PCExample_7}
\end{figure}
Given a set of two invertible\footnote{As mentionned in \cite[Remark 2.1]{AhJuJSRA}, we can assume to be dealing with invertible matrices without loss of generality in the current context.} matrices $\mset = \{A_1, A_2\}$ of dimension $n$, consider a PCLF $V \sim \pclf(\graph, \mset)$. For any pair of labels $\sigma, \sigma' \in \alphabet{2}$, we have
\begin{equation}
\forall x \in \reels^n: V_{a_7}(A_\sigma A_\sigma' x) \leq V_{a_7}(x).
\label{eq:LMIsofG7}
\end{equation}
We can show that the existence of a PCLF for this graph implies that of a PCLF for the graph $\graph_8$ in Figure \ref{fig:PCExample_8} (see e.g. \cite[Theorem 5.1]{AhJuJSRA}, \cite[Theorem 3.5]{PhEsSODT}).
\begin{figure}[!ht]
\centering
\begin{tikzpicture}[->,>=stealth',shorten >=1pt,auto,node distance=2cm,
                    semithick, scale = 1, transform shape ]
 \node[state] (V1)                           {$a_8$};
 \node[state] (V2)     [right of = V1]       {$b_8$};
  \path (V1)edge [loop,  out = 210, in = 150, min distance = 10mm]             node {$1$} (V1)
      	    edge [bend left] node {$2$}            (V2)
     	(V2)edge [bend left]             node {$1$} (V1)
      	    edge [loop,  out = 30, in = -30, min distance = 10mm] node {$2$} (V2);
\end{tikzpicture}
\caption{The Path-Complete graph $\graph_8$.}
\label{fig:PCExample_8}
\end{figure}
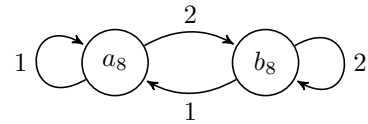
Indeed, given the PCLF $V$ above, we can construct a PCLF $U = (U_{a_8}, U_{b_8})^\top$ for $\graph_8$ by taking 
$$
\begin{aligned}
U_{a_8}(x) = V_{a_7}(x) + V_{a_7}(A_1^{-1} x), \\  U_{b_8}(x) = V_{a_7}(x) + V_{a_7}(A_2^{-1} x).
\end{aligned}
$$
This can be proven in a manner similar to that of Example \ref{example:ordering0}. 
For example, take the edge $(a_8, b_8, 2)$ in $\graph_8$, which corresponds to
$$ \forall x \in \reels^n: V_{a_7}(A_2 x) + V_{a_7}(x) \leq V_{a_7}(x) + V_{a_7}(A_1^{-1}x). $$
The inequality is satisfied by evaluating the inequality \eqref{eq:LMIsofG7} at the point $x' = A_1^{-1} x$, for $\sigma' = 1$ and $\sigma = 2$,
which gives 
$ \forall x\in \reels^n: V_{a_7}(A_2 x) \leq V_{a_7}(A_1^{-1}x). $
\label{example:extension}
\end{example}
The main difference here compared to the technique of Example \ref{example:ordering0} is that we need to generate valid inequalities involving a term $A_1^{-1} x$. Such inequalities are easily obtained through a change of variables, as done in Example \ref{example:extension}. This leads to a set of linear inequalities on a vector of the form
$$\begin{pmatrix}
V(A_2^{-1}x) \\
V(A_1^{-1}x) \\
V(x) \\
V(A_1x) \\
V(A_2x)
\end{pmatrix}. $$
By enumerating such inequalities and embedding them in a matrix representation, we can establish sufficient conditions under the form of linear programs similar to that of Theorem \ref{thm:comparisonLP} for when, given a set of matrices $\mset$ and two graphs $\graph$ and $\graph'$, the existence of a PCLF $V$ on $\graph$ for $\mset$ implies that of a PCLF $U$ on $\graph'$ where each function in $U$ is a combination of functions in $V$ and their compositions with the dynamics and their inverses up to a fixed number of compositions. Further work remains to devise conditions for when this is possible.

\section{Conclusions}

Path-Complete criteria are promising tools for the analysis of hybrid and cyber-physical systems. 
They encapsulate several powerful and popular techniques for the stability analysis of switching systems.  
Moreover, their range of application seems much wider. First, they can handle switching nonlinear systems as well and are not limited to LMIs and quadratic pieces (see \cite{AnLaASCF, AnAtALPT}). Second, they have been used to analyze systems where the switching signal is constrained \cite{PhEsSODT}. Third, they can be used beyond stability analysis \cite{PhMiDTBA,PhEsESAM,FoJuPCPO}.\\
However, many questions about these techniques still need to be clarified.  In this paper we tackle two fundamental questions. First, we gave a clear interpretation of these criteria in terms of common Lyapunov function: each criterion implies the existence of a common Lyapunov function which can be expressed as the minimum of maxima of sets of functions (Theorem \ref{thm:EmbeddedCLF}). The combinatorial structure of the graph is used to combine these sets into a CLF.
Second, we tackled the problem of deciding when one criterion is more conservative than another, and provide a first systematic approach to the problem (Theorem \ref{thm:comparisonLP}).
Note that while we focused here on Path-Complete graphs, it is clear that our approach applies to the comparison of any graphs, such as those used in the building of stability certificates for \emph{constrained switching systems} \cite{PhEsSODT}.\\
For further work, the application of the tools used in Section \ref{sec:PCLFtoCLF}, such as the observer graph, to refine the analysis of Section \ref{sec:Comparison} remains to be investigated.
On a more practical side, the common Lyapunov function representation presented in Section \ref{sec:PCLFtoCLF} could allow us to better leverage Path-Complete techniques for reachability  and invariance analysis \cite{AtSmISAF}.

\label{sec:DiscussionAndConclussion}

\bibliographystyle{IEEEtranS}
\bibliography{biblio}

\end{document}